\documentclass[12pt,letterpaper]{amsart}
\usepackage{url}
\usepackage{fullpage}
\usepackage{pslatex} % use times roman
\usepackage{amsmath}
\usepackage{amssymb}
\usepackage{amsthm}
\usepackage{color}
\usepackage{array}
\usepackage{xy}
\usepackage{setspace}
\usepackage{multicol}
\usepackage{hyperref}
\usepackage[pdftex]{graphicx}

\newcommand{\id}{\operatorname{id}}

\newtheoremstyle{slanted}% <name>
{3pt}% <Space above>
{3pt}% <Space below>
{\slshape}% <Body font>
{}% <Indent amount>
{\bfseries}% <Theorem head font>
{.}% <Punctuation after theorem head>
{.5em}% <Space after theorem heading>
{}% <Theorem head spec (can be left empty, meaning `normal')>
\theoremstyle{slanted}
\newtheorem{thm}{Theorem}[section]
\newtheorem{lem}[thm]{Lemma}
\newtheorem{prop}[thm]{Proposition}

\newtheorem{defn}[thm]{Definition}
\theoremstyle{remark}
\newtheorem{rem}[thm]{Remark}

\begin{document}

\date{} \title[New Results on Doubly Adjacent Pattern-Replacement
  Equivalences]{New Results on Doubly Adjacent Pattern-Replacement
  Equivalences} \author{William Kuszmaul} \maketitle

\begin{abstract}
In this paper, we consider the family of pattern-replacement
equivalence relations referred to as the "indices and values adjacent"
case. Each such equivalence is determined by a partition $P$ of a
subset of $S_c$ for some $c$. In 2010, Linton, Propp, Roby, and West
posed a number of open problems in the area of pattern-replacement
equivalences. Five, in particular, have remained unsolved until now,
the enumeration of equivalence classes under the $\{123,
132\}$-equivalence, under the $\{123, 321\}$-equivalence, under the
$\{123, 132, 213\}$ equivalence, and under the $\{123, 132, 213,
321\}$-equivalence. We find formulas for three of the five
equivalences and systems of representatives for the equivalence
classes of the other two. We generalize our results to hold for all
replacement partitions of $S_3$, as well as for an infinite family of
other replacement partitions. In addition, we characterize the
equivalence classes in $S_n$ under the $S_c$-equivalence, finding a
generalization of Stanley's results on the $\{12, 21\}$-equivalence.

To do this, we introduce a notion of confluence that often allows one
to find a representative element in each equivalence class under a
given equivalence relation. Using an inclusion-exclusion argument, we
are able to use this to count the equivalence classes under
equivalence relations satisfying certain conditions.
\end{abstract}

\newpage
\section{Introduction}

In this paper, we consider the family of pattern-replacement
equivalence relations referred to as the "indices and values adjacent"
case in \cite{LPRW}. Each such equivalence is determined by a
partition $P$ of a subset of $S_c$ for some $c$. Before introducing
our results, we provide background definitions.

\begin{defn}
 A \textbf{word} is a series of \textbf{letters}, where each
 \textbf{letter} has an integer value.
\end{defn}

\textbf{Example:} The word $13264$ has letters $1$, $3$, $2$, $6$, and
$4$.

\begin{defn}
 A \textbf{Permutation} of size $n$ is an $n$-letter word using each
 letter from $1$ to $n$ exactly once.
\end{defn}

\textbf{Example:} $14235$ is a permutation of size $5$.

\begin{defn}
$S_n$ denotes the set of permutations of size $n$.
\end{defn}

\textbf{Example:} $S_3 = \{ 123, 132, 213, 231, 312, 321 \}$.

\begin{defn} 
The word $w=w_1w_2\cdots w_c$ \textbf{forms the permutation}
$u=u_1u_2\cdots u_c \in S_c$ if for some integer $k$, $$w_1w_2\cdots
w_c = (u_1+k)(u_2+k) \cdots (u_c+k).$$
\end{defn}

As shorthand, we may just say that $w$ \textbf{forms} $u$.

\textbf{Example:} $7968$ forms the permutation $2413$. An even simpler
example is that $2413$ forms the permutation $2413$.

\begin{defn} 
A \textbf{replacement partition} is a partition of a subset of $S_c$
for some $c$.  Each replacement partition $P$ determines an
equivalence relation which we refer to as the
\textbf{$P$-equivalence}.
\end{defn}

\begin{defn} 
A \textbf{hit} is a contiguous subword of a permutation that forms $u$
for some $u \in S_c$. A \textbf{$P$-hit} is a contiguous subword of a
permutation that forms a permutation in a partition $P$ of $S_c$.
\end{defn}

\textbf{Example:} The subword $7968$ of the permutation $157968324$
forms the permutation $2413$.

\begin{defn} 
Given a pattern-replacement partition $P$, we define the
$P$-equivalence on $S_n$ in the following manner. Given a $P$-hit $h$
in a permutation $w$ that forms the permutation $t \in P$, we are
allowed to rearrange the letters within $h$ in any way such that the
resulting word forms some $t'$ in the same part of $P$ as $t$. A
\textbf{class}, or \textbf{equivalence class}, containing $w$ is the
set of permutations that can be reached from $w$ by repeated
rearrangements of the type just described. If two permutations $a$ and
$b$ are in the same equivalence class, we say that $a$ is
\textbf{equivalent} to $b$, or that $a\equiv b$.
\end{defn}

\begin{defn} 
Let $P$ be a replacement partition. Let $h$ be a $P$-hit in a
permutation $w$ such that $h$ forms a permutation in the $i$th part of
$P$. Rearranging the letters in $h$ in any way such that $h$ still
forms a permutation in the $i$th part of $P$ is referred to as a
\textbf{$P$-rearrangement}.
\end{defn}

\textbf{Example:} Let $w=1432657$. Then applying a $\{123, 213,
321\}\{312, 231\}$-rearrangement to the hit $432$ allows us to
rearrange the hit either as $234$ or as $324$. This results in the
permutations $1234657$ and $1324657$ respectively.

\begin{defn} 
The \textbf{$P$-equivalence} is the reflexive-transitive closure of
$P$-rearrangement. In other words, the permutations that can be
reached from a permutation $w$ by means of repeated $P$-rearrangements
are considered \textbf{equivalent} to $w$ under the
$P$-equivalence. The set of permutations equivalent to $w$ comprises
an \textbf{equivalence class}.
\end{defn}

\textbf{Example:} In Figure \ref{fig1} we show a visual depiction of the equivalence
class containing $123456$ under the $\{123, 321\}$-equivalence. Each
of the above permutations are considered \textbf{equivalent} under the
relation. Two permutations are connected by a line segment if one can
be reached from the other by the rearrangement of a single $\{123,
321\}$-hit. For example, $125436$ connects to $123456$ because $543$
forms the permutation $321$ and thus can be rearranged to form the
permutation $321$, becoming $345$.

\begin{figure}\label{fig1}
\begin{center}
\includegraphics[scale=1.3]{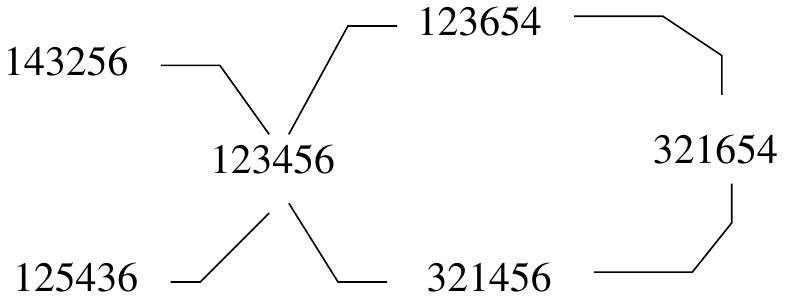}
\end{center}
\caption{Equivalence class containing $123456$ under the $\{123, 321\}$-equivalence.}
\end{figure}

We will refer to the type of equivalence relations studied in this
project as \textbf{pattern-replacement} equivalence relations. In the
literature, they are referred to either as doubly adjacent
pattern-replacement equivalence relations or as the ``indices and
values adjacent'' case.

In 2010, Linton, Propp, Roby, and West posed a number of open problems
in the area of pattern-replacement equivalences \cite{LPRW}. Five, in
particular, have remained unsolved until now, the enumeration of
equivalence classes under the $\{123, 132\}$-equivalence, under the
$\{123, 321\}$-equivalence, under the $\{123, 132, 213\}$ equivalence,
and under the $\{123, 132, 213, 321\}$-equivalence.

In Section \ref{secapplications}, we find formulas for three of the
five equivalences and systems of representatives for the equivalence
classes of the other two. We generalize our results to hold for all
replacement partitions of $S_3$, as well as for an infinite family of
other replacement partitions.

In order to do this, we first introduce the notion of
$(P,C)$-confluence in Section \ref{secconfluence}; when
$(P,C)$-confluence is satisfied by a partition $P$, one can easily
find a set of representative permutations for the equivalence classes
in $S_n$ under the $P$-equivalence. We construct a set of tools
allowing us to often quickly prove that a pattern-replacement
equivalence satisfies this confluence. Finally, we use the prior
results in the section to present a formula for the number equivalence
classes in $S_n$ under pattern-replacement equivalences satisfying
certain conditions.

In Section \ref{secSc}, we use our results to completely characterize
the equivalence classes in $S_n$ under the $S_c$-equivalence. Our
results connect interestingly to Stanley's results on the $\{12,
21\}$-equivalence \cite{S12}.

For this entire paper, fix $n$ and $c$ be positive integers such that
$c \leq n$. We will be considering equivalence classes of $S_n$ under
doubly adjacent pattern-replacement equivalences involving patterns of
size $c$.

Before continuing, we provide some useful background on the notion of
a binary relation being \textbf{confluent}.

\begin{defn}
Let $\rightarrow$ be a binary relation on a finite set $A$ and
$\rightarrow'$ be the reflexive-transitive closure of
$\rightarrow$. We say that $\rightarrow$ is \textbf{confluent} if
\begin{itemize}
\item there is no $a,b \in A$ such that $a\rightarrow' b$ and $b
  \rightarrow' a$;
\item for each connected component $C$ of $\rightarrow$ considered as
  a directed graph (with $A$ as the set of nodes), there is a unique
  minimal node $m \in C$. (i.e., there is a unique node $m \in C$ such
  that for all $c \in C$ with $c \neq m$, we have that $c \rightarrow
  ' m$.)
\end{itemize}
\end{defn}

The following lemma, known as the Diamond Lemma, is well known, and is
easy to prove.

\begin{lem}
Let $\rightarrow$ be a binary relation on a finite set $A$ such that
$a \rightarrow a$ is false for $a \in A$. Assume the following is
true.
\begin{enumerate}
\item There does not exist an infinite (possibly repeating) sequence
  $a_1, a_2, \ldots$ such that for all $i$, $a_i \in A$, and $a_i
  \rightarrow a_{i+1}$. In other words, $\rightarrow$
  \textbf{terminates}.
\item For all $a,b,c \in A$ such that $a \rightarrow b$ and $a
  \rightarrow c$, there exists $d \in A$ such that $b \rightarrow ' d$
  and $c \rightarrow ' d$ where $\rightarrow '$ denotes the
  reflexive-transitive closure of $\rightarrow$.
\end{enumerate}
Then, $\rightarrow$ is confluent.
\end{lem}

\section{General Results on a New Type of Confluence}\label{secconfluence}

In this section, we introduce the notion of a replacement partition
$P$ being $(P,C)$-confluent. When a partition is $(P,C)$-confluent, it
is easy to characterize a set of permutations in $S_n$ exactly one of
which is in a given equivalence class in $S_n$ under the
$P$-equivalence. Such a permutation is referred to as a $(P,C)$-root
permutation. We provide three results which together allow one to
often determine quickly that a pattern-replacement equivalence
satisfies $(P,C)$-confluence (Proposition \ref{propsimpleconf},
Theorem \ref{thmcombineconf}, Theorem \ref{thmomni}). We then present
a formula for the number of equivalence classes in $S_n$ under the
$P$-equivalence when $P$ is $(P,C)$-confluent and satisfies certain
conditions (Theorem \ref{thmconfcount}).

\begin{defn}  
Let $w$ be a word, each letter of which has a distinct value. The
\textbf{tail size} of $w$ is the smallest positive integer $k$ such
that the first $k$ letters of $w$ contain the $k$ smallest letter
values in $w$.
\end{defn}

\textbf{Example:} The tail size of $14238576$ is $4$ because $1423$
contains $1$, $2$, $3$, and $4$.

\begin{defn}
Let $w$ be a word, each letter of which has a distinct value.
\begin{itemize}
\item $w$ is \textbf{right leaning} if the tail size of $w$ is less
  than $n$.
\item $w$ is \textbf{left leaning} if the tail size of a written
  backwards version of $w$ is less than $n$.
\item $w$ is \textbf{omni leaning} if it is neither left leaning nor
  right leaning.
\end{itemize}
\end{defn}

\textbf{Example:} $14238576$ is right leaning because its tail size is
$4$. Consequently, writing it backwards to get $67583241$ gives us a
left leaning permutation.

\textbf{Example:} One can check that $4213$ is omni leaning.

The sets of right leaning, left leaning, and omni leaning permutations
in $S_n$ respectively are denoted by $R_n$, $L_n$, and $O_n$. Note
that the three sets do not intersect.

%% \begin{defn}
%% Let $w\in S_n$, and $h$ be a hit in $w$. Let $w'$ be $w$ with the letter $-1$ appended on its left and the letter $-2$ appended on its right. Let $a$ and $c$ be the letters before and after $h$ respectively in $w'$. Let $b$ be the average letter value in $h$.

%% Relabel $a$, $b$, and $c$ from least to greatest as $1$, $2$, and
%% $3$. If the relabeling of $a,b,c$ is in the left column of the
%% following table then $h$ is \textbf{left polarized}; if the
%% relabeling of $a,b,c$ is in the right column then $h$ is
%% \textbf{right polarized}.  \begin{center} \begin{tabular}{|l|l|}
%% \hline \textbf{left polarized} & \textbf{right polarized} \\ \hline
%% $3,2,1$ & $1,2,3$ \\ $2,1,3$ & $3,1,2$ \\ $1,3,2$ & $2,3,1$
%% \\ \hline \end{tabular} \end{center} \end{defn}

\begin{defn}
Let $w\in S_n$, and $h$ be a hit in $w$. If $h$ comprises all of $w$
we say $w$ is \textbf{left polarized}. Otherwise, let $a$ be the
average value of the letters in $h$, and let $w'$ be $w$ except with
the entire hit $h$ replaced by a single letter whose value is $a$. Let
$b$ be the letter in $w'$ closest to $a$ in value out of the (either
one or two) letters immediately adjacent to $b$ in $w'$. If there is a
tie, than pick the one on $b$'s left, although one can check it does
not matter which we pick.

If $a$ and $b$ are in increasing order in $w'$, then $h$ is
\textbf{right polarized}. If $a$ and $b$ are in decreasing order in
$w'$, then $h$ is \textbf{left polarized}.
\end{defn}

\textbf{Example:} The hit $456$ in the permutation $w=1456237$ is left
polarized. Here, $a=\frac{4+5+6}{3}=5$ and thus $w'=15237$. The two
letters immediately adjacent to $a$ in $w'$ are $1$ and $2$. Since $2$
is closer to $5$ than is $1$, $b=2$. In $w'$, $a$ and $b$ are in
decreasing order, meaning $h$ is left polarized.

\begin{defn}
A hit $h$ is \textbf{backward} in a permutation $w$ if $h$ either
\begin{itemize}
\item is both right polarized and left leaning,
\item or is both left polarized and right leaning,
\item or is omni leaning.
\end{itemize}
Otherwise, $h$ is \textbf{forward}.
\end{defn}

\begin{defn}
Let $P$ be a partition of a subset $S$ of $S_c$. Let $C \subseteq
S$. Then $C$ is a \textbf{straightening set} of $P$ if
\begin{itemize}
\item for each part of $P$ containing at least one left leaning
  permutation, $C$ contains exactly one left leaning permutation from
  that part;
\item for each part $C$ of $P$ containing at least one right leaning
  permutation, $C$ contains exactly one right leaning permutation from
  that part;
\item and for each part of $P$ containing only omni leaning
  permutation, $C$ contains exactly one permutation from that part.
\end{itemize}
\end{defn}

\textbf{Example:} $\{12345, 21435, 54231\}$ is a straightening set of
$\{12345, 12354\}\{21435, 54231, 43251\}$.

\begin{defn}
Let $C$ be a straightening set of a partition $P$. A $P$-hit $h$ in $w
\in S_n$ is \textbf{$(P,C)$-straightened} if either
\begin{itemize}
\item $h$ is forward and is in $C$,
\item or $h$ is backward and is in $C$, and there is no
  $P$-rearrangement that when applied to $h$ results in $h$ being
  forward.
\end{itemize}
In turn, \textbf{$(P,C)$-straightening} a $P$-hit $h$ is defined as
rearranging a $P$-hit to be $(P,C)$-straightened by means of a
$P$-rearrangement.
\end{defn}

\textbf{Example:} Let $P=\{123, 213, 321\}$ and $C=\{123,
321\}$. Consider the permutation $85671234$. The hit $567$ is right
leaning but left polarized and is thus backwards. Consequently,
$(P,C)$-straightening it rearranges it to be the forward hit $765$,
bringing us to the permutation $87651234$. On the other hand, $(\{123,
213, 321\}, \{123, 321\})$-straightening the hit $765$ in this new
permutation does nothing at all.

%% \begin{defn} Let $\rightarrow$ be a binary relation on a finite set
%% $A$. We say that $\rightarrow$ is \emph{confluent} if for each
%% connected component $C$ of the relation considered as a graph,
%% there is a unique element $a \in C$ such that \begin{itemize} \item
%% for all $b \in A$, $a\rightarrow b$ is false; \item for all $c \in
%% A$ satisfying $c\neq a$, $c \rightarrow ' a$ where $\rightarrow '$
%% denotes the reflexive transitive closure of
%% $\rightarrow$.  \end{itemize} \end{defn}

\begin{defn} 
Let $P$ be a partition of a subset of $S_c$ and $C$ be a straightening
set of $P$. Let the binary relation $\rightarrow$ on $S_n$ be such
that for $w,w'\in S_n$, $w \rightarrow w'$ exactly when $w'$ can be
reached from $w$ by $(P,C)$-straightening a hit in $w$ that was not
already $(P,C)$-straightened.  We refer to $\rightarrow$ as the
\textbf{$(P,C)$-straightening operator}.
\end{defn}

\begin{defn}
Let $P$ be a partition of a subset of $S_c$ and $C$ be a straightening
set of $P$. The partition $P$ is said to be \textbf{$C$-confluent} if
the $(P,C)$-straightening operator is confluent.
\end{defn}

\begin{defn}
Let $P$ be $C$-confluent. A permutation $w$ is said to be a
\textbf{$(P,C)$-root permutation} if it is the unique permutation in
its equivalence class under the $P$-equivalence such that it can be
reached from any other equivalent permutation by means of repeated
$(P,C)$-straightenings.
\end{defn}

\textbf{Example:} It turns out that $\{123, 321\}$ is $\{321,
123\}$-confluent. Figure \ref{fig2} demonstrates this for a particular
equivalence class in $S_4$. An arrow between two permutations means
one can be reached from the other under the $\{123, 321\}$-equivalence
by a $(\{123, 321\}, \{321, 123\})$-straightening. Observe how all
paths lead to the permutation $123456$ which is consequently the
$(\{123, 321\}, \{321, 123\})$-root permutation in the class.

\begin{figure}\label{fig2}
\begin{center}
\includegraphics[scale=1.3]{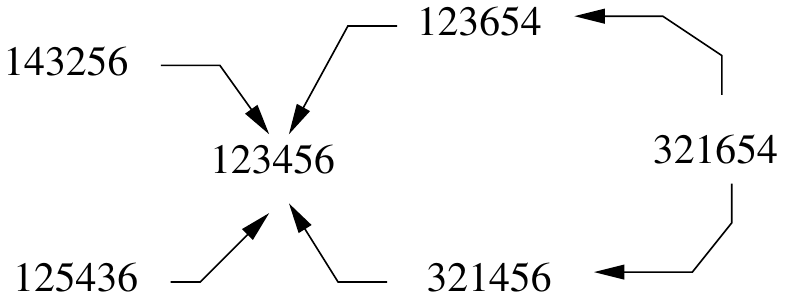}
\end{center}
\caption{Demonstration that $\{123, 321\}$ is $\{321, 123\}$-equivalent for equivalence class containing $123456$.}
\end{figure}

\begin{prop}\label{propsimpleconf}
Let $P$ be a partition of $R_c$ with parts $P_1, P_2, \ldots$. Let $C$
be the straightening set of $P$ containing the lexicographically
smallest permutation from each part of $P$. Suppose that for any two
hits $h$ and $h'$ in any permutation $w \in S_n$ such that $h'$ forms
a permutation in $P_r$, $(P,C)$-straightening $h$ leaves the letters
in $h'$ still in a contiguous subword forming a permutation in
$P_r$. Then $P$ is $C$-confluent.
\end{prop}
\begin{proof}
Let $\rightarrow$ be the $(P,C)$-straightening operator. Since $a
\rightarrow b$ implies that $b$ is lexicographically smaller than $a$,
$\rightarrow$ terminates. Given $w \in S_n$, let $D_i(w)$ be the
permutation $w'$ reached by $(P,C)$-straightening the $P$-hit
beginning with $w_i$ in $w$ if such a $P$-hit exists and $w$
otherwise. Note that $D_j(D_i(w))=D_j(D_i(D_j(w)))$ due to the
restrictions we have imposed. Thus by the Diamond lemma, $P$ is
$C$-confluent.
\end{proof}

We state the following lemmas without proof for the sake of brevity.

\begin{lem} \label{lemoverlappinghit}
Let $w \in S_n$ and $h$ and $h'$ be overlapping hits of size $c$ in
$w$. Then, $h$ and $h'$ are both forward and either are both left
leaning or are both right leaning.
\end{lem}

\begin{lem}\label{lemletterorder}
Let $w$ and $w'$ be equivalent under the $P$-equivalence where $P$
partitions a subset of $S_c$. Suppose $w_i < w_j$ for $|i-j|\ge
c$. Then, $w'_i < w'_j$.
\end{lem}

\begin{lem}\label{lemlocalorderchange}
Let $w,w' \in S_n$ be such that $w'$ can be reached from $w$ by a
$P$-rearrangement (where $P$ partitions a subset of $S_c$). Let $i$
and $j$ be such that $w_i$ is not in $h$. Then, $(w_i < w_j) = (w'_i <
w'_j)$.
\end{lem}

Now we introduce a slightly less obvious Lemma.

\begin{lem}\label{lemignorebackwardhits}
Let $U$ be a $D$-confluent partition. Then
\begin{enumerate}
\item and $(U,D)$-straightening any forward hit in a permutation
  results in a permutation with the same number of $U$-hits that are
  both backward and non-$(U,D)$-straightened.
\item $(U,D)$-straightening a non-$(U,D)$-straightened backward hit in
  a permutation results in a permutation with one fewer $U$-hits that
  are both backward and non-$(U,D)$-straightened;
\end{enumerate}
\end{lem}
\begin{proof}
Observe that (a) a backward $U$-hit cannot overlap any other $U$-hit
(by Lemma \ref{lemoverlappinghit}), and (b) $(U,D)$-straightening one
$U$-hit cannot change whether another $U$-hit is right polarized or
left polarized (by Lemma \ref{lemlocalorderchange}).

By (a) and (b), $(U,D)$-straightening a forward $U$-hit in a
permutation does not change the set of backward $U$-hits in the
permutation. Also by (a) and (b), $(U,D)$-straightening a backwards
non-$(U,D)$-straightened hit cannot rearrange another
$(U,D)$-straightened backward hit to become non-$(U,D)$-straightened.
\end{proof}

\begin{defn} 
Let $J,K,P$ be partitions of $A,B,C \subseteq S_c$ respectively such
that $A \cap B=\{\}$ and $A \cup B=C$. We say that $P$ is a
\emph{disjoint union} of $J$ and $K$ if each part of $P$ is a disjoint
union of some (possibly empty) part from $J$ and some (possibly empty)
part from $K$.
\end{defn}

\textbf{Example:} $\{1234, 2134, 3214\} \{ 4123\} \{2134, 2314\}$ is a
disjoint union of $\{1234\}\{4123\}\{2134\}$ and $\{2134,
3214\}\{2314\}$.

\begin{thm}\label{thmcombineconf}
Let $P$ be a $C$-confluent partition of $R_c$ and $Q$ be a
$C'$-confluent partition of $L_c$. Let $D=C \cup C'$ and $U$ be a
disjoint union of $P$ and $Q$. Then $U$ is $D$-confluent.
\end{thm}
\begin{proof}
Let $\rightarrow$ acting on $S_n$ be the $(U,D)$-straightening
operator and $\rightarrow'$ be the reflexive-transitive closure of
$\rightarrow$. Since $D$ is clearly a straightening set of $U$, it is
sufficient to show that $\rightarrow$ is confluent. To do this, we
will use the Diamond Lemma.

We begin by proving that $\rightarrow$ terminates. Suppose there
exists an infinite chain $a_1 \rightarrow a_2 \rightarrow \ldots$. For
each $i$, let $h_i$ be a $U$-hit $h$, the $(U,D)$-straightening of
which brings one from $a_i$ to $a_{i+1}$. Without loss of generality,
an infinite number of $i$ satisfy that $h_i$ is right leaning in $a_i$
(since otherwise, an infinite number of $i$ would satisfy that $h_i$
is left leaning).

Let $k$ be the number of backward $U$-hits in $a_1$. By Lemma
\ref{lemignorebackwardhits} there are at most $k$ values of $i$ such
that $h_i$ is backward in $a_i$. Let $j$ be the greatest such $i$.

Observe that for $i > j$, $h_i$ is forward in $a_i$. By Lemma
\ref{lemoverlappinghit}, the action of $(U,D)$-straightening a right
leaning $U$-hit commutes with the action of $(U,D)$-straightening a
left leaning $U$-hit. Recall that there exists an infinite number of
$i$ such that $h_i$ is right leaning in $a_i$, and thus also an
infinite number of such $i$ where $i>j$. It follows that for a given
value of $l$, there exists an infinite chain $b_1 \rightarrow b_2
\rightarrow \ldots$ such that $b_{i+1}$ is reached from $b_i$ by means
of the $(U,D)$-straightening of a forward right leaning
$U$-hit. However, since such hits are also $P$-hits and $P$ is
$C$-confluent, the pigeonhole principle implies this cannot be true
for $l>n!$, a contradiction.

We will now consider condition (2) of the Diamond Lemma. Suppose
$a,b,c \in S_n$ are such that $a \rightarrow b$, $a \rightarrow c$,
$a$ reaches $b$ by means of the $(U,D)$-straightening of the $U$-hit
$h$, and $a$ reaches $c$ by means of the $(U,D)$-straightening of the
$U$-hit $h'$. If $h$ and $h'$ do not overlap in $a$, then
$(U,D)$-straightening $h$ and then $(U,D)$-straightening $h'$ in $a$
yields the same permutation as does $(U,D)$-straightening $h'$ and
then $(U,D)$-straightening $h$. Suppose instead that $h$ and $h'$
overlap. By Lemma \ref{lemoverlappinghit}, we can assume without loss
of generality that both are forward and right leaning. By the
assumption that $P$ is $C$-confluent, it follows that there exists $d$
such that $b \rightarrow 'd$ and $c \rightarrow ' d$.

The result follows from the Diamond Lemma.
\end{proof}

We now show that omni leaning permutations can be nicely inserted into
arbitrary parts of partitions without screwing up the confluence
properties of the partition.

\begin{thm}\label{thmomni}
Let $P$ be a $C$-confluent partition of a subset $S$ of $S_c$. Let $o
\in O_c$ such that $o$ is not in $S$.
\begin{enumerate}
\item Let $P'$ be $P$ except with $o$ added to one of the parts of
  $P$. Then $P'$ is $C$-confluent.
\item Let $P'$ be $P$ except with $o$ added as a part of size
  $1$. Then $P'$ is $C \cup \{o\}$-confluent.
\end{enumerate}
\end{thm}
\begin{proof}
We only bother to prove (1), as (2) is straightforward. Suppose $P'$
is $P$ except with $o$ added to one of the parts of $P$. Let
$\rightarrow$ acting on $S_n$ be the $(P',C)$-straightening
operator. We will show that $\rightarrow$ satisfies the conditions of
the Diamond Lemma.

Any $P'$-hit forming $o$ is backwards. Thus, by Lemma
\ref{lemignorebackwardhits}, if there exists an infinite chain $a_1
\rightarrow a_2 \rightarrow \cdots$, then there is some $k$ such that
$a_k \rightarrow a_{k+1} \rightarrow \cdots$ and no $P'$-hits forming
$o$ are used to go from $a_i$ to $a_{i+1}$ in the chain for $i \ge
k$. But since $P$ is $C$-confluent, no such chain can exist. Thus
$\rightarrow$ terminates.

Observe that a $P'$-hit forming $o$ cannot overlap any other $P'$-hit
(Lemma \ref{lemoverlappinghit}). Since $P$ is $C$-confluent, it
follows that condition (2) of the Diamond Lemma is satisfied by
$\rightarrow$. Thus $\rightarrow$ is confluent.
\end{proof}

\begin{thm}\label{thmconfcount}
Let $P$ be a $C$-confluent partition of $S_c$. Suppose that for
$P$-hits $h$ and $h'$ in $w \in S_n$ to overlap implies that at least
one of $h$ and $h'$ is $(P,U)$-straightened. Pick $n \in \mathbb{N}$
and let $k=n!-|P|$ where $|P|$ is the number of parts in $P$. Then the
number of equivalence classes in $S_n$ under the $P$-equivalence,
which we will denote as $f(n)$, is
$$\sum\limits_{j \ge 0} \frac{(-1)^j(n-cj+j)!^2 k^j}{j!(n-cj)!}.$$
\end{thm}
\begin{proof}
Let $T_i$ be the number of $w\in S_n$ such that $w$ contains a $P$-hit
$h$ such that the first letter of $h$ is in position $i$ in $w$ and
$h$ is not $(P,C)$-straightened in $w$. By the inclusion-exclusion
principle,
$$f(n)=n! + \sum\limits_{j>0} (-1)^j \sum\limits_{S \subseteq [n],
  |S|=j} | \bigcap_{i \in S} T_i |.$$

Note that for any $P$-hit $h$ in $w \in S_n$ such that $h$ is not
$(P,C)$-straightened, there are exactly $n!-|P|$ possibilities for $h$
given the position of its first letter, the value of its smallest
letter, and whether $h$ is right polarized or left polarized. Let us
calculate the number of ways to construct $j$ non-overlapping $P$-hits
and place them in a permutation so that none of them are
$(P,C)$-straightened.

%% First, we pick the set containing exactly the $j$ letters that will each be the smallest letter in one of our $j$ hits. We claim such sets are in bijection with binary words of size $n-j(c-1)$ letters, $j$ of which are ones. Indeed, given such a word, we may substitute
Let $\Gamma$ be the set containing exactly the $j$ letters that will
each be the smallest letter in one of our $j$ $P$-hits. We claim there
are $n-j(c-1) \choose j$ possibilities for $\Gamma$. Since no two
non-$(P,C)$-straightened $P$-hits can overlap, the possibilities are
exactly the sets of $j$ integers from $1$ to $n-c+1$ such that no two of
them are within $c-1$ of each other in value. These sets, in turn, are
in bijection with binary words containing $n-j(c-1)$ letters, $j$ of
which are ones. Indeed, given such a word $w$, we may substitute each
$1$ with the letter $1$ followed by $c-1$ copies of the letter $2$,
bringing us to $w'$. We then choose to stick the letter $i$ in
$\Gamma$ exactly when the $i$th letter of $w'$ is $1$. The inverse
bijection is easy to see.
 
%% The smallest valued letter in each of the $j$ $P$-hits can be chosen in any of $n-j(c-1) \choose j$ ways. (since non-$(P,C)$-straightened hits cannot overlap, and thus we should consider each pick to represent a choice of $c$ consecutive letters).
Given $\Gamma$, the relative positions of the first letters of each of
the $j$ $P$-hits along with the letters not in any of the $j$ $P$-hits
can be chosen in $(n-j(c-1))!$ ways. This, in turn, determines which
hits are right polarized and left polarized. Given this choice, the
arrangement of each individual hit can be chosen from $n!-|P|$
possibilities. Hence $$\sum\limits_{S \subseteq [n], |S|=j} | \bigcap_{i
  \in S} T_i | = {n-j(c-1) \choose j} (n-j(c-1))! (n!-|P|)^j.$$

Inserting this into the formula found by the inclusion-exclusion
principle and slightly rearranging yields the desired result.
\end{proof}

\section{Some Immediate Applications of our Results} \label{secapplications}

In this section, we apply our results from Section \ref{secconfluence}
to all replacement partitions of $S_3$ (Theorem \ref{thmS_3case}), on
the way solving three open problems of Linton, Propp, Roby, and West
\cite{LPRW}. We then apply our results to a particularly interesting
infinite family of pattern-replacement equivalences (Theorem
\ref{thmeasyinf}).

\begin{thm}\label{thmS_3case}
(A) Let $P$ be a partition of $S_3$. Then there exists a straightening
  set $C$ of $P$ such that $P$ is $C$-confluent. (B) Furthermore, if
  $123$, $132$, and $213$ are not all in the same part of $P$ and
  $321$, $312$, and $231$ are not all in the same part of $P$, then
  the number of equivalence classes in $S_n$ under the $P$-equivalence
  can be established with Theorem \ref{thmconfcount}.
\end{thm}
\begin{proof}
By Proposition \label{propsimpleconf}, $\{123, 132\}\{213\}$ is
$\{123, 213\}$-confluent, $\{123, 213\}\{132\}$ is $\{123,
132\}$-confluent), $\{213, 132\}\{123\}$ is $\{213, 123\}$-confluent,
and $\{123, 213, 132\}$ is $\{123\}$-confluent. (These are all the
partitions of $R_3$; note that $O_3$ is empty.) Note that only in the
final of these cases does the partition not satisfy the requirements
for Theorem \ref{thmconfcount} to be applied. By symmetry, similar
statements can be said for each partition of $L_3$. Since every
partition of $S_3$ is the disjoint union of a partition of $L_3$ and a
partition of $R_3$, (A) follows from Theorem
\ref{thmcombineconf}. Since two hits of size three of size $3$ in a
permutation $w$ can only overlap if either both are in $L_3$ or both
are in $R_3$ (Lemma \ref{lemoverlappinghit}), (B) follows as well.
\end{proof}

\begin{rem}
Note that Theorem \ref{thmS_3case} resolves three open problems,
the enumeration of equivalence classes under each of the $\{123,
132\}$-equivalence, the $\{123, 321\}$-equivalence, and the $\{123,
132, 321\}$-equivalence. It also makes progress on two additional open
problems, the enumerations of equivalence classes under the $\{123,
132, 213\}$-equivalence and $\{123, 132, 213, 321\}$-equivalence; in
each case, Theorem \ref{thmS_3case} allows one to characterize a
set of permutations ($(P,C)$-root permutations), exactly one of which
appears in each equivalence class in $S_n$.
\end{rem}

\begin{thm}\label{thmeasyinf}
Let $a_1,a_2,a_3,\ldots,a_k \in L_c \cup O_c$ and
$b_1,b_2,b_3,\ldots,b_k \in R_c \cup O_c$ such that $a_i \neq b_j$ for
all $i,j$. Let $P$ be the partition $\{a_1, b_1\} \{a_2, b_2\} \cdots
\{a_k, b_k\}$. The number of equivalence classes in $S_n$ under the
$P$-equivalence is
$$\sum\limits_{j \ge 0} \frac{(-1)^j(n-cj+j)!^2 k^j}{j!(n-cj)!}.$$
\end{thm}
\begin{proof}
It is trivial that $\{a_1\}\{a_2\}\{a_3\}\cdots \{a_k\}$ is $\{a_1,
a_2, a_2, \ldots, a_k\}$-confluent and that $\{b_1\} \{b_2\} \{b_3\}
\cdots \{b_k\}$ is $\{b_1,b_2,b_3,\ldots, b_k\}$-confluent. By Theorem
\ref{thmcombineconf}, $\{a_1, b_1\} \{a_2, b_2\} \cdots \{a_k, b_k\}$
(which we will refer to as $P$) is $\{a_1, b_1, a_2, b_2, a_3, b_3,
\ldots, a_k, b_k\}$-confluent. Observe that if two $P$-hits overlap,
they either form $a_i$ and $a_j$ for some $i,j$, or they form $b_i$
and $b_j$ for some $i,j$ (by Lemma \ref{lemoverlappinghit}). In either
case, both $P$-hits are forward and thus $(P, \{a_1, b_1, a_2, b_2,
a_3, b_3, \ldots, a_k, b_k\})$-straightened. Thus, pretending that all
the permutations in $S_c$ not in $P$ are each in a part of size one of
$P$, we can apply Theorem \ref{thmconfcount} to obtain the desired
formula.
\end{proof}

\begin{rem}
Observe that Theorem \ref{thmeasyinf} counts equivalence classes in
$S_n$ under the $P$ equivalence for at least $\sum\limits_{j} { c!/2
  \choose j } {c!/2 \choose j} j!$ choices of $P \subseteq S_c$ for a
given $c$. This is easy to see by allowing only half of the omni
leaning permutations in $S_c$ to be assigned to the $a_i$ and allowing
only the other half to be assigned to the $b_i$ (so the choices for
$a_i$ come from exactly half of the permutations in $S_c$, and
similarly for $b_i$).
\end{rem}

\begin{rem}
It should be surprising that equivalence relations with different
class structures often have the same number of classes in $S_n$.

For example, by Theorem \ref{thmS_3case}, for all partitions $P$
of $S_3$ of the form $\{a,b\} \{c,d\} \{e\} \{f\}$ the $P$-equivalence
yields the same enumeration. Yet it is easy to see that the structure
of the equivalence classes created under such equivalence differs
greatly between equivalence relations.

In fact, Theorem \ref{thmeasyinf} provides an infinite number of
examples of this phenomenon.
\end{rem}

\section{The $S_c$-equivalence}\label{secSc}

One infinite family of equivalence relations in particular, is worth
studying in greater depth. The $S_c$ equivalence has strictly fewer
equivalence classes than any equivalence defined using another
replacement partition of $S_c$. Consequently, enumerating the
equivalence classes under the $S_c$-equivalence is an important
tasks. In this section, we make progress on this by completely
characterizing the equivalence classes. In the case of the $\{12,
21\}$-equivalence ($c=2$), Stanley \cite{S12} enumerated and
characterized the equivalence classes in $S_n$ and found
representative elements for each one. Surprisingly, in our work, we
will find a set of representative elements (the $(S_c, \{123\cdots c,
c \cdots 321\})$-root permutations) that in the case of the $\{12,
21\}$-equivalence differ from the already found one.

\begin{defn}
We will refer to the permutation $123 \cdots c$ as $\alpha$ and to the
permutation $c(c-1)(c-2)\cdots 1$ as $\hat{\alpha}$.
\end{defn}

\textbf{Example:} When $c=3$, the $S_c$-equivalence is the $\{123,
132, 213, 231, 312, 321\}$-equivalence, $\alpha=123$, and
$\hat{\alpha}=321$.

\begin{lem}\label{lemstayinR}
For any two $S_c$-hits $h$ and $h'$ in a permutation $w \in S_n$ such
that $h'$ forms a permutation in $R_c$, $(R_c, \{ \alpha \}
)$-straightening $h$ leaves the letters in $h'$ still forming a
permutation in $R_c$.
\end{lem}
\begin{proof}
This is because rearranging the letters of a contiguous subword of a
right leaning $S_c$-hit to be in increasing order will still yield a
right leaning $S_c$-hit.
\end{proof}

\begin{lem}\label{lemhitmaintaining}
For any two $S_c$-hits $h$ and $h'$ in a permutation $w \in S_n$,
$(S_c, \{\alpha, \hat{\alpha}\})$-straightening $h$ leaves the letters
in $h'$ still forming a permutation in $S_c$
\end{lem}
\begin{proof}
By lemma \ref{lemoverlappinghit}, we only need to consider the case
where $h$ and $h'$ are either both left leaning or both right
leaning. Without loss of generality, we can assume they are both right
leaning (and since they overlap, are both in $R_c$). The Lemma thus
follows from Lemma \ref{lemstayinR}.
\end{proof}

\begin{thm}\label{thmScrep}
The $S_c$-equivalence is $( \alpha, \hat{\alpha})$-confluent.
\end{thm}
We provide two proofs of this.
\begin{proof}
By Lemma \ref{lemstayinR} and Proposition \ref{propsimpleconf}, the
$R_c$-equivalence is $\{ \alpha \}$-confluent. By symmetry, the
$L_c$-equivalence is $\{\hat{\alpha}\}$-confluent. It follows from
Theorem \ref{thmcombineconf} that the $(L_c \cup R_c)$-equivalence is
$\{\alpha ,\hat{\alpha} \}$-confluent. It follows from Theorem
\ref{thmomni} that the $S_c$-equivalence is $\{\alpha ,\hat{\alpha}
\}$-confluent.
\end{proof}

The following alternative proof is an interesting application of the
Diamond Lemma.
\begin{proof}
Let $\rightarrow$ be the $(S_c, \{\alpha,
\hat{\alpha}\})$-straightening operator.

Observe that $(S_c, \{\alpha, \hat{\alpha}\})$-straightening a
$S_c$-hit in a permutation $w$ does not take any other $S_c$-hits that
were $\{\alpha, \hat{\alpha}\})$-straightened and rearrange them to
not be so. (Observation (1))

As a consequence of Observation (1), $\{\alpha,
\hat{\alpha}\})$-straightening a $S_c$-hit in a permutation $w$
results in a permutation $w'$ with strictly more $\{\alpha,
\hat{\alpha}\})$-straightened $S_c$-hits. Since the number of such
hits a permutation can have is bounded, $\rightarrow$ terminates.

As an additional consequence of Observation (1) and Lemma
\ref{lemhitmaintaining}, the actions of $(S_c, \{\alpha,
\hat{\alpha}\})$-straightening the $S_c$-hit beginning in position $i$
commutes with the action of $(S_c, \{\alpha,
\hat{\alpha}\})$-straightening the $S_c$-hit in position $i'$ for any
given $i$ and $i'$. Thus $\rightarrow$ satisfies condition (2) of the
Diamond Lemma. By the Diamond Lemma, $\rightarrow$ is confluent.
\end{proof}

\begin{thm}\label{thmsumclasses}
For a given equivalence class of $S_n$ under the $S_c$-equivalence,
let $w$ be the unique $( \{S_c\}, \{ \alpha, \hat{\alpha}\})$-root
permutation in the class. (Such a $w$ exists by Theorem
\ref{thmScrep}.) Then $w$ is of the form $v_1v_2\cdots v_k$ such that
\begin{enumerate}
\item each $v_i$ is an increasing or decreasing word of consecutive
  integers.
\item for all $u \equiv w$, $u$ is of the form $v'_1v'_2 \cdots v'_k$
  such that $v'_i \equiv v_i$.
\end{enumerate}
\end{thm}
\begin{proof}
Pick $v_1, v_2, \ldots, v_k$ such that $w=v_1v_2\dots v_k$, such that
each $v_i$ is an increasing or decreasing word of consecutive
integers, and such that $k$ is minimal. (A little thinking shows that
there is only one such pick.) Observe that in $w$, there is no
$S_c$-hit $h$ containing letters from both $v_i$ and $v_{i+1}$ for any
$i$. If there were, then the $S_c$-hit would have to be $(S_c,
\{\alpha, \hat{\alpha}\})$-straightened (as are all $S_c$-hits in
$w$). But if this were the case, then $v_i$ and $v_{i+1}$ would
concatenate to form an increasing or decreasing word of consecutive
integers, a contradiction because we picked $v_1, v_2, \ldots, v_k$ to
minimize $k$.

By Lemma \ref{lemhitmaintaining} and Theorem \ref{thmScrep}, it
follows that for $u \equiv w$, there are no $S_c$-hits containing both
letters from $v_i$ and from $v_{i+1}$ for any $i$. If there were, then
since repeated $(S_c, \{\alpha, \hat{\alpha}\})$-straightenings can
bring us from $u$ to $w$, $w$ would have to contain a $S_c$-hit
containing both letters from $v_i$ and $v_{i+1}$, a contradiction.
\end{proof}

\begin{defn} 
Given $w \in S_n$, we define $v_i(w)$ to be the unique $v_i$ defined
in the proof of Theorem \ref{thmsumclasses}.
\end{defn}

Observe that Stanley found a result similar to Theorem
\ref{thmsumclasses} in the case of the $\{12\}\{21\}$-equivalence
\cite{S12}.

By the previous theorem, all that remains in characterizing the
equivalence classes in $S_n$ under the $S_c$-equivalence is to
characterize the equivalence class containing $123\cdots t$ and the
one containing $t(t-1)\dots 1$ for all $t$; by symmetry, it is
sufficient to characterize the equivalence class containing the
identity permutation in $S_n$, $123\cdots n$.

\begin{defn} 
Let $w$ be a word of size $n$, each letter of which has a distinct
value. Recall that the \textbf{tail size} of $w$ is the smallest
positive integer $k$ such that the first $k$ letters of $w$ contain
the $k$ smallest letter values in $w$.

We say that if $w$ is empty, it has no \textbf{irreducible blocks}.

Otherwise, let $k$ be the tail size of $w$, let $a$ be a word made of
the first $k$ letters of $w$, and let $b$ be a word made of the final
$n-k$ letters of $w$. We say that the \textbf{irreducible blocks} of
$w$ are $B_1(w)=a$ and $B_i(w)=B_{i-1}(b)$ for $i>1$ satisfying that
$B_{i-1}(b)$ exists.
\end{defn}

\begin{defn}
A permutation is $c$-toothed if the following is true.
\begin{enumerate}
\item For each $B_i(w)$, $|B_i(w)|<c$ (where $|B_i(w)|$ is the number
  of letters in $B_i(w)$);
\item there is some sequence of consecutive integers $I=i_1, \ldots,
  i_t$ such that $\sum_{i \in I} |B_i(w)| =c.$
\end{enumerate}
\end{defn}

\textbf{Example:} Figure \ref{fig3} is a visual depiction of the irreducible block
decomposition of the permutation $w=3124657$. Note that $|B_1(w)|=3$,
$|B_2(w)|=1$, $|B_3(w)|=2$, and $|B_4(w)|=1$. Consequently, $w$ is
$3$-toothed, $4$-toothed, and $7$-toothed. However, $w$ is not
$2$-toothed since $|B_1(w)|>2$.
\begin{figure}\label{fig3}
\begin{center}
\includegraphics[scale=1.3]{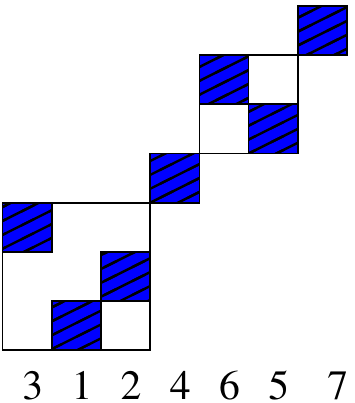}
\end{center}
\caption{Irreducible block decomposition of $3124657$.}
\end{figure}

\begin{thm}\label{thmtooth}
Let $w \in S_n$. Then $w \equiv \id_n$ exactly if $w$ is
$c$-toothed\footnote{We follow the convention that $\id_n =123\cdots
  n$.}.
\end{thm}
\begin{proof}
Observe that every $c$-toothed permutation in $S_n$ except for the
identity contains at least one $S_c$-hit not forming $\alpha$. Also
note that $( S_c, \{\alpha, \hat{\alpha}\})$-straightening a $S_c$-hit
in a $c$-toothed permutation in $S_n$ is exactly the same as
rearranging the $S_c$-hit to form $\alpha$. Finally, observe that
rearranging a $S_c$-hit in a $c$-toothed permutation to form $\alpha$
yields another $c$-toothed permutation. Thus it follows from Theorem
\ref{thmScrep} that every $c$-toothed permutation in $S_n$ is
equivalent to $\id_n$ under the $S_c$-equivalence.

It is straightforward to check that every permutation in $S_n$ that is
equivalent to $\id_n$ under the $S_c$-equivalence must also be
$c$-toothed.
\end{proof}

\begin{defn}
 Let $T_{c,n}$ be the set of $c$-toothed permutations in $S_n$. Let
 $T_c$ be the set $\{ |T_{c,j}| : j>0\}$.
\end{defn}

 It follows from Theorem \ref{thmtooth} and Theorem
 \ref{thmsumclasses} that the size of any equivalence class under the
 $S_c$-equivalence is a product of elements of $T_c$. Consequently, it
 is interesting to study $|T_{c,n}|$. We do so for $c=3$.

\begin{thm}
For $n\ge 3$, $|T_{3,n}|$ is the value of the coefficient of $x^n$ in
the generating function
$$\dfrac{x}{1-x-x^2-3x^3}-\dfrac{1}{1-x^2}.$$
\end{thm}
\begin{proof}
Let $F_n$ be the number of permutations in $S_n$ comprising only
irreducible blocks of size $\leq 3$. If the final irreducible block is
of size $1$, the final letter of such a permutation is $n$. If the
final irreducible block is of size $2$, the final two letters are
$n(n-1)$. If the final irreducible block is of size $3$, the final
three letters are $n(n-1)(n-2)$, or $n(n-2)(n-1)$, or
$(n-1)n(n-2)$. Thus $F_n=F_{n-1}+F_{n-2}+3F_{n-3}$ for $n>1$,
$F_{n<1}=0$, and $F_1=1$. Thus if $F$ is the generating function
satisfying $[x^n]F=F_n$, then $F=xF+x^2F+3x^3F+x$. Rearranging
yields $$F=\dfrac{x}{1-x-x^2-3x^3}.$$

Now consider the permutations in $S_n$ comprising only of irreducible
blocks of size $\leq 3$ but containing no consecutive irreducible
blocks whose sizes add to $3$. For $n\ge 3$, it is easy to see that
such a permutation can only exist when $n$ is even and that it is of
the form $21436587\cdots n(n-1)$ (comprising only of irreducible
blocks of size two). Thus for $n\ge 3$, $|T_{3,n}|$ is the value of
the coefficient of $x^n$ in the generating function
$$F-\dfrac{1}{1-x^2}=\dfrac{x}{1-x-x^2-3x^3}-\dfrac{1}{1-x^2}.$$
\end{proof}

\section{Conclusion and Directions of Future Work}

We present three directions of future work.

\begin{enumerate}
\item Is it interesting to extend the notion of confluence introduced
  in this paper to the case where instead of $(P,C)$-straightening
  $P$-hits, one $(P,C)$-straightens contiguous subwords of some fixed
  size. When this size of $c$, this reduces to the notion of
  $(P,C)$-straightening we have used in this paper.
\item Is there a formula counting the number of classes in $S_n$ under
  the $S_c$-equivalence? Such a formula would provide a lower bound
  for the number of equivalence classes under an arbitrary doubly
  adjacent pattern-replacement equivalence.
\item Is there a formula for $|T_{c,n}|$? Such a formula is of
  interest since the size of any equivalence class under the
  $S_c$-equivalence is a product of elements of $T_c= \{ T_{c,j} :
  j>0\}$.
\end{enumerate}

\newpage
\pagestyle{empty}\singlespace


\begin{thebibliography}{99}\setlength{\itemsep}{2ex}\normalsize

\bibitem{Kuszmaul12} William Kuszmaul, \newblock \textit{Counting
  Permutations Modulo Pattern-Replacement Equivalences for
  Three-Letter Patterns.}  \newblock \textit{Electronic Journal of
  Combinatorics}, 20(4) (2013), \#P10.

\bibitem{KZ} William Kuszmaul and Ziling Zhou, \textit{Equivalence
  Classes in $S_n$ for Three Families of Pattern-Replacement
  Relations.} MIT PRIMES,
  2013. \url{http://web.mit.edu/primes/materials/2012/Kuszmaul-Zhou.pdf}. Under
  review by \textit{Electronic Journal of Combinatorics}.

\bibitem{S12} R. Stanley, \textit{An equivalence relation on the
  symmetric group and multiplicity-free flag $h$-vectors.}
  \url{arXiv:1208.3540}, 2012.

\bibitem{PRW} A.~Pierrot, D.~Rossin, J.~West, \emph{Adjacent
  transformations in permutations.} FPSAC 2011 Proceedings, Discrete
  Math. Theor. Comput. Sci. Proc., 2011.
  \url{http://www.dmtcs.org/dmtcs-ojs/index.php/proceedings/article/view/dmAO0167/3638}.

\bibitem{LPRW} S.~Linton, J.~Propp, T.~Roby, J.~West,
  \emph{Equivalence Relations of Permutations Generated by Constrained
    Transpositions.} DMTCS Proceedings, North America, July 2010.
  \url{http://www.dmtcs.org/dmtcs-ojs/index.php/proceedings/article/view/dmAN0168},
  also \url{arXiv:1111.3920v1}.


\bibitem{NS} J.-C.~Novelli, A.~Schilling, \emph{The Forgotten Monoid.}
  RIMS Kokyuroku Bessatsu B8 (2008), 71-83.  \url{arXiv:0706.2996v3}.

\bibitem{Va} V. Fazel-Rezai, \emph{Equivalence Classes of Permutations
  Modulo Replacements Between 123 and Two-Integer Patterns.}
  \url{arXiv:1309.4802}, 2013.

 
\bibitem{_StanleyEC1} R.~Stanley, \emph{Enumerative Combinatorics
  Volume 1.} no.~49 in Cambridge Studies in Advanced Mathematics,
  Cambridge University Press, 1999.

\end{thebibliography}
\end{document}